\documentclass{amsart}
\usepackage{amsmath, amsfonts,amssymb,amsthm,amscd,latexsym}
\usepackage{cite}
\usepackage{array}
\usepackage{mathrsfs}
\usepackage{color}
\usepackage{enumerate}
 \usepackage{lmodern}
\newtheorem{thm}{Theorem}[section]
\newtheorem{theorem}[thm]{Theorem}
\newtheorem{lem}[thm]{Lemma}

\newtheorem{cor}[thm]{Corollary}

\newtheorem{prop}[thm]{Proposition}
\newtheorem{definition}[thm]{Definition}

\newtheorem{rem}[thm]{Remark}

\newtheorem{defi}[thm]{Definition}

\newcommand{\cE}{{\mathcal E}}

\newcommand{\cH}{{\mathcal H}}

\newcommand{\cM}{{\mathcal M}}
\newcommand{\cN}{{\mathcal N}}

\newcommand{\cP}{{\mathcal P}}

\newcommand{\cR}{{\mathcal R}}

\newcommand{\cZ}{{\mathcal Z}}

 \usepackage{color}
 \newcommand{\norm}[1]{\left\lVert#1\right\rVert}
\usepackage{mathtools}
\usepackage{hyperref}                   % Reference
\hypersetup{colorlinks,%                % Reference color setup
    linkcolor=blue,%
    citecolor=blue}
\usepackage{etoolbox}

\usepackage{tikz}
\usetikzlibrary{positioning}
\usepackage{caption}
 \usepackage{float}

\usetikzlibrary{petri}
\tikzset{place/.style = {circle, draw=blue!50, fill=blue!20, thick, minimum size=0.6cm},
    transition/.style = {rectangle, draw=black!50, fill=black!20, thick, minimum width=0.6cm,
                        minimum height = 1cm},
    pre/.style =    {<-, semithick},
    post/.style =   {->, semithick}
}
\usetikzlibrary{decorations.pathreplacing}
\begin{document}

 \title[Isometric embeddings of noncommutative  $L_p$-spaces]{Isometric embeddings of noncommutative  $L_p$-spaces into noncommutative symmetric spaces}

\author[J. Huang]{Jinghao Huang}
\address{Institute for  Advanced Study in  Mathematics of HIT, Harbin Institute of Technology, Harbin 150001, China}
\email{{\color{blue}jinghao.huang@hit.edu.cn}}

\author[M. Junge]{Marius Junge}
\address{Department of Mathematics, University of Illinois at Urbana-Champaign, Urbana, IL 61801,
USA}
\email{\color{blue} mjunge@illinois.edu}

 \author[F. Sukochev]{Fedor Sukochev}
  \author[D. Zanin]{Dmitriy Zanin}
 \address{School of Mathematics and Statistics, The University of New South Wales, Kensington, NSW 2052, Australia}
 \email{{\color{blue}f.sukochev@unsw.edu.au}}
 \email{{\color{blue}d.zanin@unsw.edu.au}}

 \thanks{J. Huang was supported the NNSF of China (No.12031004 and 12301160). F. Sukochev and D. Zanin were supported by the ARC}



\begin{abstract}
 We establish a noncommutative version of a familiar  Johnson--Maurey--Schechtman--Tzafriri Theorem, by showing that 
 for any $0 <p<2$ and a (not necessarily semifinite)  von Neumann algebra $\cM$ on a separable Hilbert space, if 
a symmetric quasi-Banach function space $E(0,1) $ containing the function $t\mapsto t^{-1/p}$, $0<t\le 1, $ then    there exists a noncommutative probability space  $(\cN,\sigma )$ such that 
 $L_p(\cM )$ is isometric to a subspace of 
  $E(\cN,\sigma)$. 
  In particular, this answers two questions  raised  by Randrianantoanina in 2006. 
\end{abstract}

\maketitle
 \section{Introduction} 

Identifying  subspaces  of Banach spaces lies at the core of the study of Banach space geometry, 
 which only has been achieved for very special classical function spaces.  
 The discovery of $q$-stable random variables of Kadec\cite{Kadec} resulted in the embedding   $\ell_p\hookrightarrow L_q(0,1)$ when $1\le q<  p<2$, where $X\hookrightarrow Y$ means that a Banach space $X$ is isomorphic to a subspace of another Banach space $Y$. 
In 1966,  Bretagnolle, Dacunha-Castelle and Krivine \cite{BDK,BD} obtained the following result  by using sequences of independent $r$-stable random variables:% and ultrapowers of Banach spaces:
  \begin{quote}
  for $1\le p<r<2$, the space $L_r  (0,1)\hookrightarrow L_{p}(0,1)$   isometrically.
  \end{quote}
 %Finite dimensional results were obtained with combinatorical tools
%by Kwapien and Sch\"{u}tt\cite{KS1,KS2} and extended by Raynaud and  Sch\"{u}tt \cite{RS}.
%  Johnson and Schechtman \cite{JS} studied the
%concrete problem of embedding $\ell^n_q$ into $\ell^m_p$
% nearly isometrically and started
%the investigation of the minimal number $m(\varepsilon; q; p; n)$ for $(1+ \varepsilon)$-isomorphic
%embeddings. 
%Bourgain, Lindenstrauss and Milman \cite{BLM} invented the iteration method and they were able to extend the results to arbitrary finite
%dimensional subspaces of $L_p$. 
%Moreover, they clarified the behaviour of the
%function $m(\varepsilon; q; p; n)$  for all values $0 < p < q < 2$. Talagrand \cite{LT,Ta}
% improved the estimates for embeddings of arbitrary $n$-dimensional subspaces
%of $L_p$ into $\ell_p^m$.  {\color{red}It is not necessary.  I leave it with you to decide whether we should delete it. }
  
A far-reaching  extension of  Bretagnolle, Dacunha-Castelle and Krivine's  result is  due to Johnson et. al.\cite[Section 8]{JMST} (see also \cite[Theorem 2.f.4]{LT2} and see \cite{GM} for the case when $0< p<2$).
 \begin{theorem}\label{JMST}
 Let $1<p<2$ and let $E(0,1)$ be a symmetric function space containing $t^{-1/p}$. Then, $L_p(0,\infty)\hookrightarrow E(0,1)$  isometrically. 
 \end{theorem}
 
 Since  von Neumann, Segal and Dixmier initiated the 
  study   of noncommutative $L_p$-spaces\cite{v37,Se,D53}, 
  the geometry of noncommutative $L_p$-spaces has attracted significant attention.
It is reasonable to expect similar albeit established with different techniques results for noncommutative  $L_p$-spaces, which are not isomorphic to any Banach lattices.  
Indeed, in 1967,
 McCarthy~\cite{Mc} proved that, if $1\le p<\infty $, then  the Schatten--von Neumann class $C_p$ embeds into $L_p(0,1)$ if and only if $p=2$.
Arazy and Lindenstrauss \cite{AL} showed that
for  $1\le p<\infty $, 
  $$L_p(0,1)\hookrightarrow C_p \mbox{ if and only if } p=2.$$
The complete isomorphic classification of noncommutative $L_p$-spaces associated with hyperfinite von Neumann algebras   was obtained in \cite{hrs} (see also \cite{S00,S96,S01}).
In particular, if $1\le p<\infty $, then\cite{hrs,S96} 
$$C_p \hookrightarrow  L_p(\cM,\tau) \mbox{ for some   finite  von Neumann algebra 
} (\cM,\tau) \mbox{ if and only if }  p=2. $$
Isometries between noncommutative $L_p$-spaces on semifinite von Neumann algebras have been completely characterized by Yeadon\cite{Yeadon} (see also \cite{Sherman,Watanabe} for the type $III$ case). 
We note that,  $L_p(0,\infty)$ is isometric to $L_p(0,1)$ (see \cite[Part III]{Wojtaszczyk},  \cite[Chapter XII]{Banach} and \cite{S00,S01}) while the Lorentz sequence space $\ell_{p,q}$ is not isomorphic to a  subspace of the noncommutative Lorentz space $L_{p,q}(\cM,\tau)$ over a  noncommutative probability space $(\cM,\tau)$ whenever $1<  p <\infty $, $1\le   q<\infty $ and $  p\ne q $\cite{KS,SS2018,HSSZ}. 
 Therefore, it is natural  to consider the following problem 
\begin{quote}
 \emph{whether a symmetric space on a von Neumann algebra with an infinite trace is isomorphic to a subspace of another symmetric space on 
  a noncommutative probability space. }
\end{quote}

  When $p$ is distinct from $q$, the situation for
isomorphic embedding of a noncommutative $L_p$-space
into a noncommutative $L_q$-space becomes  quite different, %.
%Results concerning  the isomorphic embeddings of  von Neumann--Schatten  classes  $C_p$ ($1<p<2$) into the predual of a von Neumann algebra were obtained in
see e.g. \cite{Junge01,JP,Rand,Xu}.
In general, the existence of such an embedding strongly relies on the underlying algebras. 
It is well-known that $$C_p\not\hookrightarrow C_q$$ whenever $1\le p\ne q<\infty $ and $p\ne 2$, see e.g.  \cite{Arazy81,AL}. 
We emphasize that in the problem stated above, we ask for symmetric space on a finite von Neumann algebra.
 Indeed,   the same argument used in \cite[p.211]{LT2} (see  Remark \ref{Remark} below, see also \cite[Theorem 0.2]{Junge01} for a similar result) yields that, 
for any $1\le q<p <2$, 
$$C_p \hookrightarrow L_q\left (B(\cH)\bar{\otimes} L_\infty (0,\infty ), \tau\otimes \int \right )$$ isometrically.  
A noncommutative version of $r$-stable random variables  was introduced in \cite{Junge01}. Using this notion, it was stated  that \emph{
for any     semifinite von Neumann algebra $\cM$ with a semifinite faithful normal trace $\tau$, there exists a finite von Neumann algebra $\cN$ with a faithful normal tracial state $\sigma$ such that for any $0<q <p<2$, $$L_p  (\cM,\tau)\hookrightarrow L_{q}(\cN,\sigma)$$ isometrically.}
This result was later extended to the setting of type $III$ von Neumann algebras in \cite{Rand}.
In recent works\cite{JSZ,HJSZ},   isomorphic embeddings of noncommutative symmetric spaces into noncommutative $L_p$-spaces were considered. The main result achieved in these papers may be succinctly formulated as follows: 
\emph{for 
 a  symmetric (Banach) function space  $E(0,\infty) $  on $(0,\infty)$ such that $E(0,\infty) \cap L_{\infty}(0,\infty) \neq L_p(0,\infty)\cap L_{\infty}(0,\infty),$  $1\leq p<2 $, and  a semifinite   $\sigma$-finite   von Neumann algebra $(\mathcal{M},\tau)$,
if  $E(0,\infty)\hookrightarrow L_p(0,1)$, then 
there exists a noncommutative probability space $(\cN, \sigma )$ such that}   $$E(\cM,\tau )\hookrightarrow L_p(\cN,\sigma).$$
However, a noncommutative version of Theorem~\ref{JMST} has not been obtained yet and the main goal of the present paper is to study the isometric version of the problem stated above. 
Another motivation of the present paper is a question posed by Randrianantoanina in 2006\cite[p.2]{Rand}: {\it It is a natural question to consider whether or not Theorem 1.1 (i.e. \cite[Theorem 0.3]{Junge01})
can be extended to include general von Neumann algebras which are not necessarily semifinite.}
 The main result of the present paper is  the following   noncommutative version of Theorem~\ref{JMST}, which provides a very simple criterion for a noncommutative symmetric space to have a subspace isometric to a noncommutative $L_p$-space,  and answers Randrianantoanina's question in the affirmative. 
\begin{thm}\label{main}
 
 Let $(\mathcal{M}_0,\tau_0)$ be a  von Neumann algebra on a separable Hilbert space.   
 There exists a probability space $(\mathcal{N},\sigma)$ such that 
  $L_p(\mathcal{M}_0)$  is isometrically to a subspace of  $E(\mathcal{N},\sigma)$  whenever $0<  p<2 $ and   $E(0,1)$ is  a quasi-Banach symmetric function space with 
$t\mapsto t^{1/p}\in E(0,1) $. \end{thm}
%We quote the following from the paper  by  Randrianantoanina\cite[p.19]{Rand}: {\it
%We remark that the embedding in   Theorem 3.1  only deals with isomorphism.
%We do not know it one can obtain a similar result for isometric embedding. }
In \cite{Rand}, Randrianantoanina
 posed another  open problem\cite[p.19]{Rand}:  
\begin{quote}Let $M$ and $N$ be von Neumann algebras.
If $M$ is finite and $L^p(N)$ embeds isometrically into $L^q(M)$ for all $0<q<p<2$, is $N$ semi-finite?
\end{quote}
Theorem \ref{main} above provides a negative answer to this problem.

The approach  used in \cite{JSZ,HJSZ} relies on ultra product techniques, which does not go beyond the scope of (commutative and noncommutative) $L_p$-spaces.  
Moreover,  since the noncommutative notion of independence is not uniquely defined  in the noncommutative setting, the approach of $q$-stable random variables used by~ Kadec~\cite{Kadec} is  not applicable in the noncommutative setting.  
Therefore we
prefer to use a  more functional analytic tool, the so-called   Kruglov
operator introduced and studied in \cite{AS04,AS,AS06,AS10}, whose noncommutative version was defined and thoroughly studied   in \cite{JSZ}.

Note that surjective isometries on (commutative and noncommutative) symmetric spaces rather have elementary forms, see e.g. \cite{Kalton_R93,Kalton_R,Zaidenberg,HS,FJ,CMS,S96b,SV,HSZ}, while the injective isometries constructed in this paper are not of elementary forms (see also \cite{Junge01,HJSZ,Rand}).

 \section{Preliminaries}
 \subsection{Symmetric function spaces}

Let $I$ be $(0,1)$ or $(0,\infty)$ and let $L(I)$ denote the space of all  (real)  Lebesgue-measurable functions $x$ on $I.$
For a Lebesgue-measurable function $x$ on $I,$ we define its {\it distribution function} by the formula
$$d_x(s)=m(\{t:\ x(t)>s\}),\quad s\in\mathbb{R},$$
where $m$ stands for Lebesgue measure. 
Denote by $S(I)$ the subalgebra of $L(I)$
 consisting of all functions $x$ such that $d_{|x|}(s) < \infty$  for some 
 $s > 0$.
 
Two measurable functions $x$ and $y$ are called {\it equimeasurable} (written, $x\sim y$) if their distribution functions $d_x$ and $d_y$ coincide. In particular, for every measurable function $x\in S    (I),$ the function $|x|$ is equimeasurable with its {\it decreasing rearrangement} $\mu(x)$ defined by the formula
$$\mu(t;x):=\inf \{\tau\geq0:\ d_{|x|}(\tau)<t \},\quad t>0.$$
If $ x,y\in S(I),$ then $\mu(x)=\mu(y)$ if and only if $|x|$ and $|y|$ are equimeasurable.% We recall that a function $x$ is said to be symmetrically distributed, if $x$ and $-x$ are equimeasurable.

\begin{defi} [see e.g. \cite{KPS}] Let $X\subset S(I)$ be a quasi-Banach space.
\begin{enumerate}%[{\rm (a)}]
\item $X$ is said to be a quasi-Banach function space if, from $x\in X,$ $y\in S(I)$ and $|y|\leq |x|,$ it follows that $y\in X$ and $\left\|y\right\|_X\leq \left\|x\right\|_X.$
\item a quasi-Banach function space $X$ is said to be a symmetric if, for every $x\in X$ and any measurable function $y,$ the assumption $\mu(y)=\mu(x)$ implies that $y\in X$ and $\left\|y\right\|_X=\left\|x\right\|_X.$
\end{enumerate}
\end{defi}
Without lost of generality, in what follows we always assume that $\left\|\chi_{(0,1)}\right\|_X=1$ for any symmetric function space  $X=X(0,1)$ on $(0,1)$.

\subsection{Noncommutative symmetric spaces}For detailed exposition of material in this subsection,  we refer to \cite{DPS}.
In what follows,  $\cH$ is a  Hilbert space and $B(\cH)$ is the
$*$-algebra of all bounded linear operators on $\cH$ equipped with the uniform norm $\left\|\cdot\right\|_\infty$, and
$\mathbf{1}$ is the identity operator on $\cH$.
Let $\mathcal{M}$ be
a von Neumann algebra on $\cH$.
We denote by $\cP(\cM)$ the collection  of all projections in $\cM$, by $\cM'$ the commutant of $\cM$ and by $\cZ(\cM)$ the center of $\cM$.
For more information about von Neumann algebras, see e.g. \cite{KR,Tak,Dixmier}.

A closed, densely defined operator $x:\mathfrak{D}\left( x\right) \rightarrow \cH $ with the domain $\mathfrak{D}\left( x\right) $ is said to be {\it affiliated} with $\mathcal{M}$
if $yx\subseteq xy$ for all $y\in \mathcal{M}^{\prime }$, where $\mathcal{M}^{\prime }$ is the commutant of $\mathcal{M}$.
A  closed,
densely defined
operator $x:\mathfrak{D}\left( x\right) \rightarrow \cH $ affiliated with $\cM $ is said to be
{\it measurable}  if  there exists a
sequence $\left\{ p_n\right\}_{n=1}^{\infty}\subset \cP\left(\mathcal{M}\right)$, such
that $p_n\uparrow \mathbf{1}$, $p_n(\cH)\subseteq\mathfrak{D}\left(x\right) $
and $\mathbf{1}-p_n$ is a finite projection (with respect to $\mathcal{M}$)
for all $n$.
 The collection of all measurable
operators with respect to $\mathcal{M}$ is denoted by $S\left(
\mathcal{M} \right) $, which is a unital $\ast $-algebra
with respect to strong sums and products (denoted simply by $x+y$ and $xy$ for all $x,y\in S\left( \mathcal{M%
}\right) $).

From now on, let $\mathcal{M}$ be a
semifinite von Neumann algebra equipped with a faithful normal
semifinite trace $\tau$.

An operator $x\in S\left( \mathcal{M}\right) $ is called $\tau$-measurable if
$\tau(e^{|x|}(s,\infty))<\infty$ for sufficiently large $s$, where
by $e^{|x|}$ is denoted the spectral measure of $|x|$.
The collection $S\left( \mathcal{M}, \tau\right)
$ of all $\tau $-measurable
operators is a unital $\ast $-subalgebra of $S\left(
\mathcal{M}\right) $.

Consider the algebra $\mathcal{M}=L^\infty(0,\infty)$ of all
Lebesgue measurable essentially bounded functions on $(0,\infty)$.
The algebra $\mathcal{M}$ can be seen as an abelian von Neumann
algebra acting via multiplication on the Hilbert space
$\mathcal{H}=L^2(0,\infty)$, with the trace given by integration
with respect to Lebesgue measure $m.$
It is easy to see that the
algebra of all $\tau$-measurable operators
affiliated with $\mathcal{M}$ can be identified with
the algebra $S(0,\infty)$.

\begin{definition}\label{mu}\cite{FK,Nelson,Se}
Let $x\in
S(\mathcal{M},\tau)$. The generalized singular value function $\mu(x):t\mapsto  \mu(t;x)$ of
the operator $x$ is defined by setting
$$
\mu(s;x)
=
\inf\{ s:d_{|x|}(s)\le t \}, ~t\ge 0,
$$
where $d_b(s) :=\tau(e^b(s,\infty))$, $s\in \mathbb{R}$, $b=b^*\in S(\cM,\tau)$. 
\end{definition}

%It is well-known \cite{LSZ,DP2} that if $a\in S(\cM,\tau)$ and $b,c\in \cM$, then
%\begin{align}\label{ineqfact}
%\mu(t;bac)\le \left\|b\right\|_\infty \left\|c\right\|_\infty \mu(t;a), ~\mu(t;a^*)=\mu(t;a).
%\end{align}
%{\color{red}In particular, $\mu(A) =\mu(|A|)$.}
%{\color{red}Suppose that $X\in S(\cM,\tau)$. If $0<\alpha \in \mathbb{R}$ and $E = E^{|X|}(\alpha,\infty)$, then
%\begin{align}\label{mu_shift}
%\mu(|X|E) = \mu(X)\chi_{[0,\tau(E))}
%\end{align}
%and
%\begin{align}\label{mu_shift2}
%\mu(t; |X|E^\perp) = \mu(t+\tau(E);X)
%\end{align}
%for all $t\ge 0$ whenever $\tau(E)<\infty$.}

\begin{definition}\label{def:symmetric}\cite{DPS}
 A linear subspace $E$ of $S(\cM,\tau)$ equipped with a complete quasi-norm $\norm{\cdot}_E$, is called a quasi-Banach  symmetric space (of $\tau$-measurable operators) if $x\in S(\cM,\tau)$, $y \in E$ and $\mu(x)\le \mu(y)$ imply that $x\in E$ and $\norm{x}_E \le \norm{y}_E$.
\end{definition}

It is well-known that any quasi-Banach  symmetric space $E$ is a quasi-normed $\cM$-bimodule, that is, $axb\in E$ for any $x\in E$, $a,b\in \cM$ and $\left\|axb\right\|_E\leq \left\|a\right\|_\infty\left\|b\right\|_\infty \left\|x\right\|_E$ \cite{DP2,DPS}.
%A linear subspace $E$ of $S(\cM,\tau)$ equipped with a complete norm $\|\cdot\|_E$, is called \emph{fully symmetric space} (of $\tau$-measurable operators) if $X\in S(\cM,\tau)$, $Y \in E$ and $X\prec\prec Y$ imply that $X\in E$ and $\|X\|_E \le \|Y\|_E$.
%

%
%
%
%
%The so-called K\"{o}the dual is identified with an important part of the dual space. If $E  \subset S(\cM,\tau)$ is a symmetric space, then the K\"{o}the dual $E^\times $ of $E$ is defined by setting $$ E^\times =\{   X\in S(\cM,\tau) : \sup_{\|Y\|_E\le 1, Y\in E}\tau (|XY|)   <\infty    \}.$$
%

 A wide class of quasi-Banach  symmetric operator spaces associated with the von Neumman algebra $\cM$ can be constructed from concrete symmetric function spaces studied extensively in e.g. \cite{KPS}. Let 
 $\cE :=  E(0,\infty) $ be a  quasi-Banach  symmetric function space on the semi-axis $(0,\infty)$ (or $\cE : =E(0,1) $ for the case when  $(\cM,\tau)$ is \emph{a noncommutative probability space}, i.e., $\tau$ is a faithful normal tracial state). Then the pair 
 $$E(\cM,\tau)=\{x\in S(\cM,\tau):\mu(x)\in \cE \},\quad \left\|x\right\|_{E(\cM,\tau)}:=\left\|\mu(x)\right\|_{\cE }$$ is a  quasi-Banach  symmetric space on $\cM$ \cite{Kalton_S} (see also \cite{LSZ,S14}). 
% For convenience, we denote $\left\|\cdot\right\|_{ E(\cM,\tau)}$ by $\left\|\cdot\right\|_E $. 
%
% 

% If $\tau>0,$ the dilation operator $\sigma_{\tau}$ is defined by setting 
% $\sigma_{\tau}x(s)=x(s/{\tau}),$ $s>0,$ in the case of the semi-axis. In the case of the interval $(0,1),$ the operator $\sigma_{\tau}$ is defined by
% $$
% \sigma_{\tau}x(s)=
% \begin{cases}
% x(s/\tau),& s\leq\min\{1,\tau\}\\
% 0,& \tau<s\leq1.
% \end{cases}
% $$

%We write $B\prec\prec A$ (and say that $B$ is submajorized by $A$ in the sense of Hardy--Littlewood--P\'{o}lya) if
%$$\int_0^t\mu(s,B)ds\leq\int_0^t\mu(s,A)ds,\quad t>0.$$
%If $  A,B\in L_1(\mathcal{M},\tau)$ are positive  operators  such that $B\prec\prec A$ and $\tau(B)=\tau(A),$ then we write $B\prec A$ (and say that $B$ is majorized by $A$ in the sense of Hardy--Littlewood--P\'{o}lya).

%Recall  the following properties of submajorization (see e.g. \cite[Theorem  3.3.3 and Lemma 3.3.7]{LSZ})
%\begin{equation}\label{maj property1}
%A+B\prec\prec\mu(A)+\mu(B),\quad A,B\in (L_1+L_{\infty})(\mathcal{M},\tau)
%\end{equation}
%and 
%\begin{equation}\label{maj property2}
%A\oplus B\prec\prec A+B,\quad 0\leq A,B\in (L_1+L_{\infty})(\mathcal{M},\tau).
%\end{equation}
%If also $0\leq A,B\in L_1(\mathcal{M},\tau),$ then we can replace $\prec\prec$ with $\prec.$

 \subsection{The Kruglov operator}\label{S:k}
 Let $\Omega =\Pi_{k=0}^\infty (0,1)$ be the (infinite dimensional) hypercube equipped with the product Lebesgue measure 
$dm^\infty $. 
The Kruglov operator $K$ acts from $L_1(0,1)$ to $L_1(\Omega)$ by the following formula
$$Kx=\sum_{k=1}^\infty \sum_{m=1}^k \chi_{_{A_k}} \otimes \chi_{(0,1)}^{\otimes (m-1)} \otimes x \otimes \chi_{(0,1)}^{\otimes \infty 
}, ~x\in L_1(0,1), $$
where $A_k$, $k\ge 0$, are pairwise disjoint sets with $m(A_k) = \frac{1}{e\cdot k!} $ for all $k\ge 0$ (so that
 $\cup_{k\ge 0}
A_k =(0,1)$) and where $\chi_B$ is the indicator function of the measurable set $B\subset (0,1)$. 

Let $\cM$ be a semifinite von Neumann algebra (on a separable Hilbert space $\cH$)  equipped with a faithful normal trace $\tau$. 
The Kruglov operator $K_\cM$ is a \emph{positive} bounded  operator from  $L_1(\cM,\tau)$ into $L_1(\cN,\sigma)$ for some finite von Neumann algebra $\cN$ equipped with a faithful normal tracial state $\sigma_\cM$ such that \cite[Theorem 32]{JSZ}
$$\sigma_\cM ({\rm exp}(i K_\cM x)) ={\rm exp}(\tau({\rm exp}(ix)-1)), ~x=x^* \in L_1(\cM,\tau). $$ 
Moreover, if $\cM$ is hyperfinite, then $\cN$ can be chosen to be hyperfinite\cite[Theorem 32]{JSZ}. The operator $K_\cM$ has the following properties:
\begin{enumerate}
  \item If self-adjoint operators $x$ and $ y\in L_1(\cM,\tau)$ are equimeasurable (i.e., $d_{x_+}=d_{y_+} $ and 
  $d_{x_-}=d_{y_-}$), then $K_\cM x, K_\cM y \in L_1(\cN,\sigma_\cM)$ are equimeasurable too.
  In addition, if $x=x^*\in L_1(\cM,\tau)$ is equimeasurable with $y\in L_1 (0,1)$, then the elements $K_\cM x\in L_1(\cN,\sigma_\cM)$
  and $Ky\in L_1(0,1)$ are equimeasurable~\cite[Corollary 33(i)]{JSZ}. 
  \item The operator $K_\cM$ acts boundedly from $(L_1\cap L_2)(\cM,\tau)$ %(equipped with the norm $\norm{\cdot}_{L_1\cap L_2} =\norm{\cdot}_{L_1}+\norm{\cdot}_{L_2}$) 
  into $L_2(\cN,\sigma_\cM)$ \cite[Corollary 33(ii)]{JSZ}.
 %  \item For all $x,y \in (L_1\cap L_2)(\cN)$, the equality $[K_\cM x, K_\cM y]=K_\cM [x,y]$ holds. 
%  \item If $x_k =x_k^*\in L_1(\cM,\tau)$, $1\le k\le n$, are such that $x_j x_k=0$ for $j\ne k $, then the random variable $K_\cM x_k\in L_1(\cN,\sigma_\cM)$, $1\le k \le n$, are independent in the sense of \cite[Section 2.8]{JSZ}. 
%  \item The operator $K_\cM$ maps symmetrically distributed operators from $L_1(\cM,\tau)$ into symmetrically distributed operators from $L_1(\cN,\sigma_\cM)$. 
  \item If $\cM=\cR\bar{\otimes} B(\cH)$, then $K_\cM$ acts from $L_1(\cR\bar{\otimes} B(\cH), \tau\otimes {\rm Tr} )$ into $L_1(\cR,\tau)$~\cite[Corollary 34]{JSZ}. 
  \item For any $\tau$-finite projection $e\in \cM$, by the construction of Kruglov operators, 
  the Kruglov operator $K_{e\cM e}$ coincides with the restriction of $K_\cM$ on $e\cM e$ (up to $*$-isomorphisms), see  the proof of \cite[Theorem 32]{JSZ}. 
\end{enumerate}
%We denote by   $K_{\mathcal{M}}$    the Kruglov operator corresponding to the semifinite von Neumann algebra $(\mathcal{M},\tau).$
\subsection{The Haagerup $L_p$-spaces}\label{Haagerup}
We recall the construction of noncommutative $L_p$-spaces associated with an arbitrary 
von Neumann algebra. 
We use Haagerup's definition \cite{Ha} (see also \cite[Chapter 9]{Hiai}). 

Let $\cM$ be an arbitrary von Neumann algebra with a faithful normal semifinite weight $\phi_0$.
We consider the one-parameter modular automorphism group $\sigma^{\phi_0} = \left\{\sigma_t^{\phi_0}\right\}_{t\in \mathbb{R}}$
(associated with $\phi_0$) on $\cM$ and obtain a semifinite crossed product von Neumann algebra
$$\mathfrak{M} :=\cM \rtimes_{\sigma^{\phi_0}}\mathbb{R},$$
which admits the canonical semifinite trace $\tau$ and a trace-scaling dual action $\{\theta_s\}_{s\in \mathbb{R}}$ such that 
$$\tau\circ \theta_s =e^{-s }\tau \mbox{  for all }s\in \mathbb{R}.$$
Note that $\mathfrak{M}$ is of type $II_\infty$ when $\cM$ is non-trivial\cite[Theorem 4.7]{Daele}.

The original von Neumann algebra $\cM$ can be identified with a $\theta$-invariant von Neumann subalgebra  of $\mathfrak{M}  $. 
For $0< p<\infty $, the noncommutative Haagerup $L_p$-space $L_p(\cM)$ is defined by 
$$L_p (\cM) :=   \left\{x  \in S(\mathfrak{M}  ,\tau):\theta_s(x)=e^{-\frac{s}{p}  }x \mbox{ for all } s\in \mathbb{R}\right\}.$$
It is known that there is a linear bijection $\psi\mapsto x_\psi$ between the predual $\cM_*$ and $L_1(\cM)$. 
Due to this correspondence, we define the trace ${\rm tr}: L_1(\cM)\to \mathbb{C}$ by setting 
$${\rm tr}(x_\psi):=\psi(1),~x_\psi \in L_1(\cM). $$
For a given $x\in L_p(\cM)$, $0<p<\infty$, we have the polar decomposition $x=u|x|,$ where $|x|$ is a positive operator in $L_p(\cM) $ and $u$ is a partial isometry contained in $\cM$. The quasi-Banach norm on $L_p(\cM)$ is given by 
$$\norm{x}_{L_p(\cM)}:= {\rm  tr}(|x|^p), ~x\in L_p(\cM).$$
 Recall that  the weak $L_p$-space $L_{p,\infty}(0,\infty )$ is defined by \cite{KPS,DPS,Bennett_S}
$$L_{p,\infty }(0,\infty ) := \left\{f\in S(0,\infty): \norm{f}_{L_{p,\infty}(0,\infty)} := \sup_{0<t<\infty }\left\{t^{\frac1p }\mu(t;f)\right\} \right\}.$$
The space $L_p(\cM)$ is a closed linear subspace of $L_{p,\infty }(\mathfrak{M}  ,\tau)$
and (see e.g.  \cite[Lemma 4.8]{FK}, see also \cite[Lemma 2.4]{ST} or \cite[Lemma 9.14]{Hiai})
\begin{align}\label{haLp}\mu(t;x) =\norm{x}_{L_p(\cM)}\cdot t^{-\frac1p}, ~\forall t>0. 
\end{align}
Note that if the von Neumann algebra $\cM$ has separable predual, then $\mathfrak{M}$ also has separable predual (see e.g. \cite[Section 8.2 or p.204]{Hiai}).

\section{Proof of Theorem \ref{main}}
In the section, we prove the main theorem of the present paper,  Theorem \ref{main} above.
Before proceeding to the proof of Theorem~\ref{main}, we need several properties of the Kruglov operator. 

We denote  $(L_1\cap L_\infty) (\cM,\tau) :=  L_1(\cM,\tau)\cap\cM$, the space consists of $x\in S(\cM,\tau))$ such that 
\begin{align}\label{def:L1}
\norm{x}_{(L_1\cap L_\infty
) (\cM,\tau)}:=\max\left\{\norm{x}_{L_1(\cM,\tau)},\norm{x}_\cM \right\}<\infty.
\end{align}
We denote the set of all partitions\footnote{A partition $\pi=(V_1,V_2,\cdots)$ of a set $X$ is a collection of non-empty subsets $V_1,V_2,\cdots$ of $X$ such that every element $x$ in $X$ is in exactly one of these subsets $V_i$'s (i.e., the subsets are nonempty mutually disjoint sets).  In the present paper, we always assume that elements in $V_i$ are increasing. } of the set $\{1,\cdots,n\}$ by $S_n.$  The following lemma is the key ingredient of the proof of Theorem \ref{main} below.

\begin{lem}\label{mixed momenta formula} Let $\cM$ be a von Neumann algebra on a separable Hilbert space,  equipped with a semifinite faithful normal trace $\tau.$  We have
\begin{align}\label{31}
\sigma_{\mathcal{M}}\left(\prod_{k=1}^nK_{\mathcal{M}}(x_k)\right)=\sum_{\pi\in S_n}\prod_{V\in\pi}\tau\left(\prod_{k\in V}x_k\right),~ x_1,\cdots,x_n\in (L_1\cap L_{\infty})(\mathcal{M},\tau).
\end{align}
\end{lem}
\begin{proof} {\bf Step 1:} Assume that  $\tau({\bf 1})<\infty$, where ${\bf 1}$ stands for the identity of $\cM$. 
Recall that \cite[(28)]{JSZ}
$$K_{\mathcal{M}}x= 0 \oplus \bigoplus_{l=1}^\infty \left(\sum_{m=1}^l {\bf 1}^{\otimes (m-1)}  \otimes x \otimes {\bf 1}^{\otimes (l-m)}\right),$$
and \cite[(18)]{JSZ}
$$\sigma_\cM=e^{-\tau({\bf 1})} \bigoplus_{l=0}^\infty \frac{1}{l!}\tau^{\otimes l}.$$
We have 
\begin{align*}
\prod_{k=1}^nK_{\mathcal{M}}(x_k)&=\prod_{k=1}^n\left(0 \oplus \bigoplus_{l=1}^\infty \left(\sum_{m=1}^l {\bf 1}^{\otimes (m-1)}  \otimes x_k \otimes {\bf 1}^{\otimes (l-m)}\right)\right)\\
&=0\oplus\bigoplus_{l=1}^{\infty}\prod_{k=1}^n\left(\sum_{m=1}^l {\bf 1}^{\otimes (m-1)}  \otimes x_k \otimes {\bf 1}^{\otimes (l-m)}\right),
\end{align*}
and therefore, 
\begin{align*}
&~\quad \sigma_{\mathcal{M}}\left(\prod_{k=1}^nK_{\mathcal{M}}(x_k)\right)\\
&=\sum_{l=1}^{\infty}\frac1{e^{\tau({\bf 1})}\cdot l!} \cdot \tau^{\otimes l}\left(\prod_{k=1}^n\left(\sum_{m=1}^l {\bf 1}^{\otimes (m-1)}  \otimes x_k \otimes {\bf 1}^{\otimes (l-m)}\right)\right)\\
&=\sum_{l=1}^{\infty}\frac1{e^{\tau({\bf 1})}\cdot l!}\sum_{m_1,\cdots,m_n=1}^l\tau^{\otimes l}\left(\prod_{k=1}^n\left({\bf 1}^{\otimes (m_k-1)}  \otimes x_k \otimes {\bf 1}^{\otimes (l-m_k)}\right)\right).
\end{align*}
 For a fixed number $l\ge 1$ and a given partition  $\pi=\{V_1,V_2,\cdots, V_{|\pi|}\}$ (where $|\pi|$ stands for the cardinality of $\pi$) of the set $\{1,\cdots,n\},$ consider the following set of   $n$-tuples
$$C:= \left\{(m_1,\cdots m_n ) \in \{1,\cdots ,l\}^n : \mbox{$m_i=m_j$ iff $i,j\in V$ for some $V\in\pi$ }   \right\}.$$
In other words,
 each $n$-tuple corresponds to an injective mapping from $\{V_1,V_2,\cdots,V_{\pi }\}$ into $\{1,\cdots, l\}$, see below.
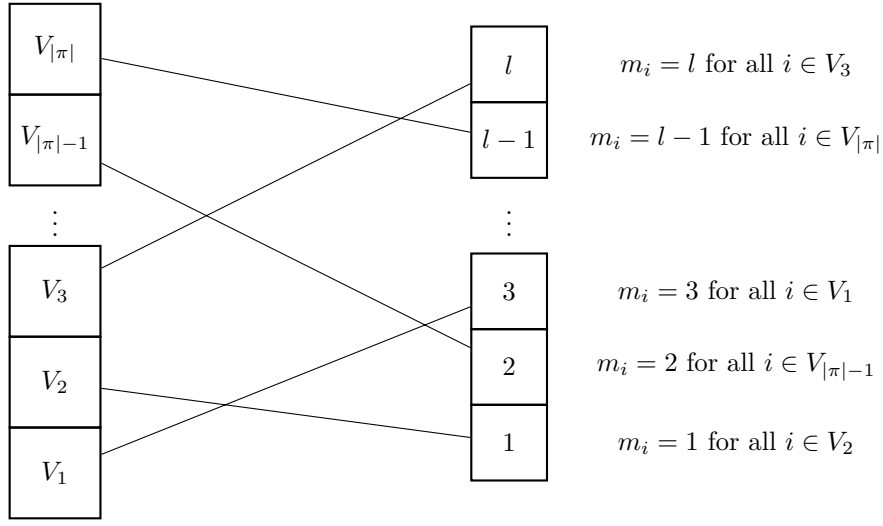
\begin{center}
\begin{tikzpicture}
  \node (a) at (0,6.2)  [rectangle,thick,draw,minimum height=1.2cm,minimum width=1.2cm]  {$V_{|\pi|}$};
  \node (b)at (0,5) [rectangle,thick,draw,minimum height=1.2cm,minimum width=1.2cm]   {$V_{|\pi|-1}$};
  \node (c) at (0,4) [rectangle]  {$\vdots$};
  \node (d)  at (0,3) [rectangle,thick,draw,minimum height=1.2cm,minimum width=1.2cm]  {$V_3$};
  \node (e)  at (0,1.8) [rectangle,thick,draw,minimum height=1.2cm,minimum width=1.2cm] {$V_2$};
  \node (f) at (0,0.6) [rectangle,thick,draw,minimum height=1.2cm,minimum width=1.2cm]  {$V_1$};
\node (g)  at (6,6) [rectangle,thick,draw,minimum height=1cm,minimum width=1cm]   {$l$};
\node  at (9,6)    {$m_i=l$ for all $i\in V_3$};
\node (h)  at (6,5)[rectangle,thick,draw,minimum height=1cm,minimum width=1cm]  {$l-1$};
\node  at (9,5)    {$m_i=l-1$ for all $i\in V_{|\pi|}$};
\node  (i) at (6,3) [rectangle,thick,draw,minimum height=1cm,minimum width=1cm]    {$3$};
\node  at (9,3)    {$m_i=3$ for all $i\in V_{1}$};
\node   (j) at (6,4 ) [rectangle]  {$\vdots$};
\node (k)  at (6,2)  [rectangle,thick,draw,minimum height=1cm,minimum width=1cm]  {$2$};
\node  at (9,2)    {$m_i=2$ for all $i\in V_{|\pi|-1}$};
\node  (l) at (6,1) [rectangle,thick,draw,minimum height=1cm,minimum width=1cm]    {$1$};
\node  at (9,1)    {$m_i=1$ for all $i\in V_{2}$};
  \draw  (a) -- (h)
    (e) -- (l)
(f) -- (i)
(d) --(g)
     (b) -- (k);
\end{tikzpicture}
\end{center}
The cardinality of $C$ is ($m_i$'s go through all possible values, see the following picture) 
 $$l\cdot (l-1) \cdots (l-|\pi |+1) = \frac{l!}{(l-|\pi|)!}, \mbox{ if } \mbox{ $l$ and $\pi $ such that  $l\ge |\pi|$;} 
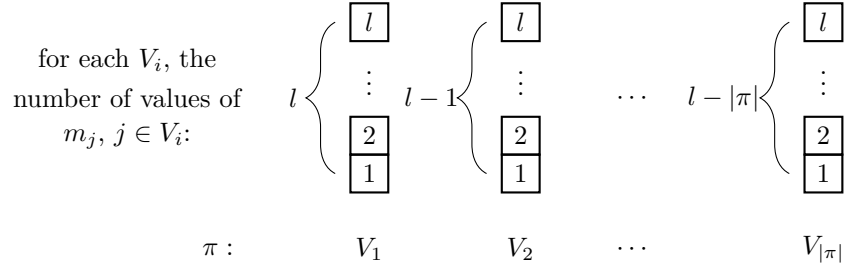
$$
\begin{center}
\begin{tikzpicture}
%Nodes
%\node[squarednode]      (maintopic)                              {2};
\node at (0,4) [rectangle,thick,draw,minimum height=0.5cm,minimum width=0.5cm] (l) {$l$};
\node at (0,3.3) [rectangle] () {$\vdots$};
\node at (0,2.5) [rectangle,thick,draw,minimum height=0.5cm,minimum width=0.5cm] () {$2$};
\node at (0,2) [rectangle,thick,draw,minimum height=0.5cm,minimum width=0.5cm] () {$1$};
\node at (0,1)  () {$V_1$};
\node at (2,4) [rectangle,thick,draw,minimum height=0.5cm,minimum width=0.5cm] (l) {$l$};
\node at (2,3.3) [rectangle] () {$\vdots$};
\node at (2,2.5) [rectangle,thick,draw,minimum height=0.5cm,minimum width=0.5cm] () {$2$};
\node at (2,2) [rectangle,thick,draw,minimum height=0.5cm,minimum width=0.5cm] () {$1$};
\node at (2,1)  () {$V_2$};
\node at (3.5,3) [rectangle] () {$\cdots$};
\node at (3.5,1) [rectangle] () {$\cdots$};
\node at (6,4) [rectangle,thick,draw,minimum height=0.5cm,minimum width=0.5cm] (l) {$l$};
\node at (6,3.3) [rectangle] () {$\vdots$};
\node at (6,2.5) [rectangle,thick,draw,minimum height=0.5cm,minimum width=0.5cm] () {$2$};
\node at (6,2) [rectangle,thick,draw,minimum height=0.5cm,minimum width=0.5cm] () {$1$};
\node at (6,1)  () {$V_{|\pi|}$};
\node at (-2,1)  () {$\pi:$};
\node at (-3.2,3.5)  () {for each $V_i$, the};
\node at (-3.2,3)  () {number of values of};
\node at (-3.2,2.5)  () {$m_j$, $j\in V_i $:};
\draw [decorate,decoration={brace,amplitude=11pt,raise=1ex}]
  (-.3,2) -- (-.3,4) node[midway,xshift=-2em]{$l$};
  \draw [decorate,decoration={brace,amplitude=11pt,raise=1ex}]
  (1.7,2) -- (1.7,4) node[midway,xshift=-2.5em]{$l-1$};
    \draw [decorate,decoration={brace,amplitude=11pt,raise=1ex}]
  (5.7,2) -- (5.7,4) node[midway,xshift=-2.9em]{$l-|\pi|$};
%\node[squarednode]      (rightsquare)       [right=of maintopic] {3};
%\node[roundnode]        (lowercircle)       [below=of maintopic] {4};

%Lines
%\draw[->] (uppercircle.south) -- (maintopic.north);
%\draw[->] (maintopic.east) -- (rightsquare.west);
%\draw[->] (rightsquare.south) .. controls +(down:7mm) and +(right:7mm) .. (lowercircle.east);
\end{tikzpicture}
\end{center}

For any  tuple $(m_1,\cdots,m_n)\in C,$ we have
\begin{align*}
&~\quad \tau^{\otimes l}\left(\prod_{k=1}^n\left({\bf 1}^{\otimes (m_k-1)}  \otimes x_k \otimes {\bf 1}^{\otimes (l-m_k)}\right)\right)\\
&=
\tau^{\otimes l}\left( \cdots   \otimes \left( \prod_{k\in V_1 }  x_k\right) \otimes   \cdots   \otimes \left( \prod_{k\in V_2 }  x_k\right) \otimes   \cdots  \otimes \left( \prod_{k\in V_3 }  x_k\right)\otimes \cdots  \right)\\
&
=\prod_{V\in\pi}\tau\left(\prod_{k\in V}x_k\right)\cdot \tau({\bf 1})^{l-|\pi|}.
\end{align*}
Observe that,
if  $l<  |\pi|$, then 
$C$ is empty, 
   i.e., 
 the cardinality $|C|$  is $0$. Thus,
\begin{align*}
&~\quad \sigma_{\mathcal{M}}\left(\prod_{k=1}^nK_{\mathcal{M}}(x_k)\right)\\
&=\sum_{l=1}^{\infty}\frac1{e^{\tau({\bf 1})}\cdot l!}\sum_{\pi\in S_n}\prod_{V\in\pi}\tau\left(\prod_{k\in V}x_k\right)\cdot \tau({\bf 1})^{l-|\pi|}\cdot \frac{l!}{(l-|\pi|)!}\chi_{_{\{|\pi|\leq l\}} }\\
&=\sum_{\pi\in S_n}\prod_{V\in\pi}\tau\left(\prod_{k\in V}x_k\right)\cdot\left(\frac1{e^{\tau({\bf 1})} }  \sum_{l= |\pi|}^\infty \frac1{ l!}\cdot\tau({\bf 1})^{l-|\pi|}\cdot\frac{l!}{(l-|\pi|)!}\right)\\
&=\sum_{\pi\in S_n}\prod_{V\in\pi}\tau\left(\prod_{k\in V}x_k\right)\cdot\left(\frac1{e^{\tau({\bf 1})} } \sum_{l= |\pi|}^\infty  \frac{\tau({\bf 1})^{l-|\pi|}}{  (l-|\pi|)!}\right)\\
&=\sum_{\pi\in S_n}\prod_{V\in\pi}\tau\left(\prod_{k\in V}x_k\right)\cdot\left(\frac1{e^{\tau({\bf 1})} }   \sum_{m=0}^\infty \frac{\tau({\bf 1})^m }{ m !}\right)\\
&=\sum_{\pi\in S_n}\prod_{V\in\pi}\tau\left(\prod_{k\in V}x_k\right).
\end{align*}

{\bf Step 2:} Now, we prove the  general case when $\tau({\bf 1})$ is not necessarily finite.
Since $K_\cM$ is continuous from $L_1(\cM,\tau)$ into $L_1(\cN,\sigma)$ and $x_k$'s are elements in $(L_1\cap L_\infty ) (\cM,\tau)$, it suffices to 
prove the case when  each $x_k$ has  a 
$\tau$-finite left (and right) support projection. Let $$p=\left(
\bigvee_{k=1}^n\mathfrak{l}(x_k)\right)\bigvee\left(\bigvee_{k=1}^n\mathfrak{r}(x_k)\right),$$
where $\mathfrak{l}(x)$ and $\mathfrak{r}(x)$ stand for the left and the right supports respectively. 
In particular, $\tau(p)<\infty$ (see e.g. \cite[Proposition 1.15.9]{DPS}). It follows from Step 1 that
$$\sigma_{\mathcal{M}}\left(\prod_{k=1}^nK_{\mathcal{M}}(x_k)\right)\stackrel{\tiny \mbox{Section } \ref{S:k}} {=} \sigma_{p\mathcal{M}p}\left(\prod_{k=1}^nK_{p\mathcal{M}p}(x_k)\right)
=\sum_{\pi\in S_n}\prod_{V\in\pi}\tau\left(\prod_{k\in V}x_k\right).$$
This completes the proof.
\end{proof}

For any $x$ in a semifinite von Neumann algebra $\cM$, we define 
$$v(x)=K_{\mathcal{M}\otimes \mathbb{M}_2(\mathbb{C})}(x\otimes E_{01}),$$
where $E_{01}=\left(
                \begin{array}{cc}
                  0 & 1 \\
                  0 & 0 \\
                \end{array}
              \right)
$ and $
E_{10}=\left(
                \begin{array}{cc}
                  0 & 0 \\
                  1 & 0 \\
                \end{array}
              \right). $

Let $n\ge 1$. 
We say that partition $\pi\in S_{2n}$ is {\it good} (denoted by $\pi\in S_{2n}^{{\rm good}}$) if, for every $V=\{k_1,\cdots,k_m\} \in\pi,$ we have $|V|=m $ is even and $k_{l+1}$ has the opposite parity to $k_l$ for every $1\leq l<m.$ Clearly, such a partition exists, e.g. $\left\{ \{1,2\},\{3,4\},\cdots,\{2n-1, 2n\}  \right\}$. 

\begin{lem}\label{lemmagood}Let $\cM$ be a von Neumann algebra on a separable Hilbert space, equipped with a semifinite faithful normal trace $\tau$.   For any $n\ge 1$, 
we have
\begin{align}\label{eqlemmagood}
\sigma_{\mathcal{M}\otimes \mathbb{M}_2(\mathbb{C})}\left (|v(x)|^{2n}\right)=\sum_{\pi\in S_{2n}^{{\rm good}}}\prod_{V\in\pi}\tau\left(|x|^{|V|}\right),~\forall x\in( L_1\cap L_\infty )(\cM,\tau).
\end{align}
\end{lem}
\begin{proof}
Since $K_{\mathcal{M}\otimes \mathbb{M}_2}$ preserves positivity, it follows that $K_{\mathcal{M}\otimes \mathbb{M}_2}$ maps self-adjoint operators to self-adjoint operators.
Hence, 
\begin{align}\label{adjointv}
v(x)^{\ast}& = \left( K_{\mathcal{M}\otimes \mathbb{M}_2(\mathbb{C})}(x\otimes E_{01})  \right)^* \nonumber \\
& = \left( K_{\mathcal{M}\otimes \mathbb{M}_2(\mathbb{C})}\left(
\frac{ 
x\otimes E_{01} +(x\otimes E_{01})^*}{2}  +  i \frac{x\otimes E_{01} - (x\otimes E_{01})^* }{2i}
 \right)  \right)^* \nonumber \\
& =   K_{\mathcal{M}\otimes \mathbb{M}_2(\mathbb{C})}\left(
\frac{ 
x\otimes E_{01} +(x\otimes E_{01})^*}{2}  -  i \frac{x\otimes E_{01} - (x\otimes E_{01})^* }{2i}
 \right)  \\
 & =   K_{\mathcal{M}\otimes \mathbb{M}_2(\mathbb{C})}\left(
 \left(
 x\otimes E_{01}
 \right)^*  
 \right) \nonumber \\
 &=K_{\mathcal{M}\otimes \mathbb{M}_2(\mathbb{C})}(x^{\ast}\otimes E_{10}).\nonumber 
 \end{align}
Letting \begin{align}\label{oddeven}
x_k=
\begin{cases}
x^{\ast}\otimes E_{10},& k\mbox{ is odd};\\
x\otimes E_{01},& k\mbox{ is even}, 
\end{cases}
\end{align}
we have 
\begin{align} \begin{split}\label{sigmaM2}
&~\quad \sigma_{\mathcal{M}\otimes \mathbb{M}_2(\mathbb{C})}(|v(x)|^{2n})\\
&=\sigma_{\mathcal{M}\otimes \mathbb{M}_2(\mathbb{C})}(( v(x)^*v(x))^{ n})\\
&\stackrel{\eqref{adjointv} }{=}\sigma_{\mathcal{M}\otimes \mathbb{M}_2(\mathbb{C})}\left( 
\left( K_{\mathcal{M}\otimes \mathbb{M}_2(\mathbb{C})}\Big (x^{\ast}\otimes E_{10}\Big ) \cdot  K_{\mathcal{M}\otimes \mathbb{M}_2(\mathbb{C})}\Big(
x^{\ast}\otimes E_{01}\Big)\right ) ^{ n} \right)\\
&\stackrel{\eqref{oddeven}}{=} \sigma_{\mathcal{M}\otimes \mathbb{M}_2(\mathbb{C})}\left(\prod_{k=1}^{2n}K_{\mathcal{M}\otimes \mathbb{M}_2(\mathbb{C})}(x_k)\right).
\end{split}
\end{align}
Applying  Lemma \ref{mixed momenta formula} by replacing $\tau$ with $\tau\otimes{\rm Tr}$ (where ${\rm Tr} $ stands for the standard trace on $\mathbb{M}_2$), we have 
\begin{align*}
\sigma_{\mathcal{M}\otimes \mathbb{M}_2(\mathbb{C})}(|v(x)|^{2n})&\stackrel{\eqref{sigmaM2}}{=} \sigma_{\mathcal{M}\otimes \mathbb{M}_2(\mathbb{C})} \left(\prod_{k=1}^{2n}K_{\mathcal{M}\otimes \mathbb{M}_2(\mathbb{C})}(x_k)\right)\\
&\stackrel{\eqref{31}}{=}\sum_{\pi\in S_{2n}}\prod_{V\in\pi}(\tau\otimes{\rm Tr})\left(\prod_{k\in V}x_k \right).
\end{align*}

We claim that 
if $V=\left\{k_1,\cdots,k_m \right\} \in \pi $ is such that
$$(\tau\otimes{\rm Tr}) \left(\prod_{k\in V}x_k\right)\neq0,$$
then $m$ is even and $k_{l+1}$ has the opposite parity to $k_l$ for every $1\leq l<m.$ 
Indeed, if $k_{l+1}$ and $k_l$ have the same parity for some 
$1\leq l<m$, then $x_{k_{l}}x_{k_{l+1}} $  is equal  to either $$  (x^{\ast}\otimes E_{10})(x^{\ast}\otimes E_{10}) \mbox{ or }(x\otimes E_{01})(x\otimes E_{01})  ,$$ both of which are $0$. 
Hence, $k_{l+1}$ has the opposite parity to $k_l$ for every $1\leq l<m.$ 
If   $m$ is odd, then 
$$\prod_{k\in V}x_k = \left(  \big(x^*\otimes E_{10})(x\otimes E_{01} \big) \right)^{\frac{m-1}{2}} (x ^* \otimes E_{10  }) = |x|^\frac{m-1}{2}x^*\otimes E_{10  }$$
$$ \qquad \mbox{ or }\left(  \big(x \otimes E_{01})(x^* \otimes E_{10 }\big ) \right)^{\frac{m-1}{2}}  (x  \otimes E_{01 }) =  |x^*|^\frac{m-1}{2}x \otimes E_{01},$$ whose traces are $0$. 
This proves the claim. 
In other words, 
if  $$(\tau\otimes{\rm Tr})\left(\prod_{k\in V_n}x_k\right)\neq0 $$ for all $n\ge 1$, then $\pi$ is good.
Hence, 
\begin{align}\label{pro1}
\sigma_{\mathcal{M}\otimes \mathbb{M}_2(\mathbb{C})} \left( |v(x)|^{2n}\right )= 
%\sigma_{\mathcal{M}\otimes \mathbb{M}_2(\mathbb{C})} \Big(\prod_{k=1}^{2n}K_{\mathcal{M}\otimes \mathbb{M}_2(\mathbb{C})}(x_k)\Big)=
\sum_{\pi\in S_{2n}^{\rm good}}\prod_{V\in\pi}\left(\tau\otimes{\rm Tr}\right )\left(\prod_{k\in V}x_k\right).
\end{align}
Let $V= (k_1,\cdots,k_{|V|})\in \pi \in S_{2n}^{\rm good}$.   If $k_1$ is odd, then
\begin{align}\begin{split}\label{proodd}
\left( \tau\otimes{\rm Tr}\right)\left(\prod_{k\in V}x_k\right)&=( \tau\otimes {\rm Tr} )\left ((x^{\ast}x)^{\frac{|V|}{2}}\otimes \left(
                \begin{array}{cc}
                  0 & 0 \\
                  0 & 1 \\
                \end{array}
              \right) \right)  \\
              &=\tau\left(
(x^{\ast}x)^{\frac{|V|}{2}}\right)=\tau\left(|x|^{|V|}\right)
\end{split}
\end{align}
and if $k_1$ is even, then
\begin{align}\begin{split}\label{proeven}
(\tau\otimes{\rm Tr})\left(\prod_{k\in V}x_k\right)&=( \tau\otimes {\rm Tr} )\left((x x^{\ast})^{\frac{|V|}{2}}\otimes \left(
                \begin{array}{cc}
                  1 & 0 \\
                  0 & 0 \\
                \end{array}
              \right)  \right)  \\
& =\tau\left((xx^{\ast})^{\frac{|V|}{2}}\right) \stackrel{\mbox{\tiny \cite[Prop. 3.4.3]{DPS}}}{=} \tau\left(x^* (xx^{\ast})^{\frac{|V|}{2}}  x\right) \\
&=\tau\left(( x^{\ast}x )^{\frac{|V|}{2}}\right)  = \tau\left( |x|^{|V|} \right). 
\end{split}\end{align}
Combining \eqref{pro1}, \eqref{proodd} and \eqref{proeven}, the assertion follows.
\end{proof}

The following estimate follows from a  well-known result concerning   Bell numbers, see e.g. \cite[Theorem 2.1]{BT} and \cite{Bruijn}. 
\begin{lem}\label{moments of poisson} There exists $0<c\in\mathbb{R}$ such that
$$\sum_{\pi\in S_n}1\leq (cn)^n,\quad n\in\mathbb{N}.$$
\end{lem}
%\begin{proof} Let $K$ be the (commutative) Kruglov operator. We have
%$$\int (Kf)^n=\sum_{\pi\in S_n}\prod_{V\in\pi}\int f^{|V|}.$$
%Setting $f=\chi_{(0,1)},$ we obtain
%$$\int (K\chi_{(0,1)})^n=\sum_{\pi\in S_n}1.$$
%Since $K:\exp(L_1)(0,1)\to\exp(L_1)(0,1),$ it follows that $K\chi_{(0,1)}\in \exp(L_1)(0,1).$ Since the Orlicz norm in $\exp(L_1)(0,1)$ is equivalent to the norm
%$$f\to\sup_{n\in\mathbb{N}}\frac1n\|f\|_{L_n(0,1)},$$
%it follows that
%$$\|K\chi_{(0,1)}\|_{L_n(0,1)}=O(n),\quad n\in\mathbb{N}.$$
%Taking the $n$-th power, we complete the proof.
%\end{proof}

 The following results follows from a theorem  due to Carleman. %{\color{blue}It should be either in Sur les fonctions indefiniment derivables, C. R. Acad. Sci Paris, 177 (1923), 422--424 or in Les fonctions quasi analytiques, Collection de Monographies sur la Theorie des Fonctions, Gauthier--Villars, Paris, 1926. FIND PRECISE REFERENCE.}
 \begin{thm}[see e.g. \cite{Sjodin} or \cite{C1,Carleman}]\label{carleman theorem} Let $f_1$ and $f_2$ be measurable functions on $(0,1).$ If there exists a constant $c>0$ such that 
$$\int_0^1f_1^n dm =\int_0^1f_2^n dm =O((cn)^n),\quad n\in\mathbb{N},$$
then $f_1$ and $f_2$ are equimeasurable.
\end{thm} 

In what follows, we denote  $$r:=\chi_{(0,\frac12]}-\chi_{(\frac12,1]}.$$

\begin{lem}\label{lemmav1} Let $\cM$ be a  von Neumann algebra on a separable Hilbert space, equipped with a semifinite faithful normal trace $\tau$. 
If $x,y\in (L_1\cap L_{\infty})(\mathcal{M},\tau)$ are such that $\mu(x)=\mu(y),$ then $$\mu(v(x))=\mu(v(y)).$$
\end{lem}
\begin{proof} Consider two measurable functions $f_1$ and $f_2$ on $(0,1)$ defined by  formulas
$$f_1(t)=\mu\left(\{2t\};v(x)\right)\cdot r(t),\quad f_2(t)=\mu\left( \{2t\}; v(y)\right)\cdot r(t),\quad t\in (0,1).$$
Here, $\{\cdot\}$ denotes the fractional part. 

If $n$ is odd, then 
\begin{align*}
\int_{(0,1)}f_1^n dm &  =\int_0^1 \mu\left(\left\{2t\right\};v(x)\right)^n \cdot r(t) dt \\
&=\int_0^\frac12 \mu\left(2t;v(x)\right)^n  dt -\int_\frac12^1 \mu\left(2t-1;v(x)\right)^n  dt  \\
&=\frac12 \int_0^1 \mu\left( t;v(x)\right)^n  dt -\frac12 \int_0^1 \mu\left(t;v(x)\right)^n  dt   
=0 
\end{align*}
and, similarly, 
$$\int_{(0,1)}f_2^ndm=0. $$
If $n$ is even, then it follows from Lemma \ref{lemmagood} (applied with $\frac{n}{2}$ instead of $n$) that
\begin{align*}
\int_{(0,1)}f_1^ndm  &~\qquad =~\qquad\int_0^1\mu(\{2t\};v(x))^n dt \\
&~\qquad=~\qquad\int_0^\frac12 \mu\left(2t;v(x)\right)^n  dt +\int_\frac12^1 \mu\left(2t-1;v(x)\right)^n  dt  \\
&~\qquad=~\qquad\frac12 \int_0^1 \mu\left( t;v(x)\right)^n  dt +\frac12 \int_0^1 \mu\left(t;v(x)\right)^n  dt   \\
&~\qquad=~\qquad\int_0^1\mu (t;v(x))^ndt \\
&~\qquad=~\qquad\sigma_{\mathcal{M}\otimes \mathbb{M}_2(\mathbb{C})}(|v(x)|^n)\\
&\qquad \stackrel{\eqref{eqlemmagood}}{=} \qquad\sum_{\pi\in S_n^{{\rm good}}}\prod_{V\in\pi}\tau(|x|^{|V|})\\
& \stackrel{\mbox{\tiny \cite[Prop.3.4.5]{DPS}}}{\le} \sum_{\pi\in S_n^{{\rm good}}}\prod_{V\in\pi}\norm{ x }_\cM^{|V|-1}  \tau(|x|)\\
&\qquad  \stackrel{\eqref{def:L1}}{\leq} \qquad\sum_{\pi\in S_n^{{\rm good}}}\prod_{V\in\pi}\left\|x\right\|_{ (L_1\cap L_{\infty})(\mathcal{M},\tau)}^{|V|}\\
&~\qquad=~\qquad\sum_{\pi\in S_n^{{\rm good}}}\left\|x\right\|_{ (L_1\cap L_{\infty})(\mathcal{M},\tau)}^n\\
&~\qquad\leq~\qquad \Big(\sum_{\pi\in S_n}1\Big)\cdot \left\|x\right\|_{ (L_1\cap L_{\infty})(\mathcal{M},\tau)}^n,
\end{align*}
and, similarly, \begin{align*}
\int_{(0,1)}f_2^ndm   \leq  \Big(\sum_{\pi\in S_n}1\Big)\cdot \left\|y\right\|_{ (L_1\cap L_{\infty})(\mathcal{M},\tau)}^n,
\end{align*}
By Lemma \ref{moments of poisson}, we have 
$$\int_{(0,1)}f_1^n dm =\int_{(0,1)}f_2^n dm =O((c\left\|x\right\|_{ (L_1\cap L_{\infty})(\mathcal{M},\tau) }n)^n),\quad n\in\mathbb{N}.$$
It follows now from Theorem \ref{carleman theorem}
that $f_1$ and $f_2$ are equimeasurable.
By the definitions of $f_1$ and $f_2$, we have  $\mu(v(x))=\mu(v(y)).$ The proof is complete. 
\end{proof}

Recall that $(L_p+L_2)(0,\infty )$ is the separable symmetric quasi-Banach function space defined by 
$$(L_p+L_2)(0,\infty ):=\left\{f \in S(0,\infty): \norm{f}_{L_p+L_2}:= \inf \{\norm{h}_{L_p} +\norm{g}_{L_2}:f=h+g\} \right\}.$$

Lemma \ref{lemmav1} above only treats  operators in $(L_1\cap L_{\infty})(\mathcal{M},\tau)$. 
The following lemma should be compared with Lemma \ref{lemmav1}, which allows us to consider operators in $(L_p+L_2)(\mathcal{M}_1,\tau_1 )$. 
\begin{lem}\label{lemmav} Let $\cM_1$ be a  von Neumann algebra on a separable Hilbert space, equipped with a semifinite faithful normal trace $\tau_1$.   
Let $\mathcal{M}=\mathcal{M}_1\bar{\otimes}L_{\infty}(0,1)$ and $\tau=\tau_1\otimes\int.$ 
If $0<p\le 2$, then, for every $x,y\in (L_p+L_2)(\mathcal{M}_1,\tau_1 )$ with $\mu(x)=\mu(y),$ we have $$\mu(v(x\otimes r))=\mu(v(y\otimes r)),$$
where $v(z)=K_{\mathcal{M}\bar{\otimes} \mathbb{M}_2(\mathbb{C})}(z \otimes E_{01})$, $z \in (L_p+L_2)(\cM,\tau)$. 
\end{lem}
\begin{proof} 
 
Note that 
$\cM \bar{\otimes} \mathbb{M}_2(\mathbb{C}) = \cM_1\bar{\otimes}  L_{\infty}(0,1)  \bar{\otimes} \mathbb{M}_2(\mathbb{C})  $. 
By Proposition \ref{propbounded},  the restriction $K_{\mathcal{M}\bar{\otimes} \mathbb{M}_2(\mathbb{C})} (\cdot \otimes r ) $ of   
the Kruglov operator $K_{\mathcal{M}\bar{\otimes} \mathbb{M}_2(\mathbb{C})}   $ is bounded from 
 $(L_p+L_2) (\mathcal{M}_1 \bar{\otimes} \mathbb{M}_2 (\mathbb{C}),\tau_1\otimes {\rm Tr })$ to $L_p(\mathcal{N}_{\mathcal{M}\bar{\otimes} \mathbb{M}_2(\mathbb{C})},\sigma_{\mathcal{M}\bar{\otimes}  \mathbb{M}_2(\mathbb{C})})$. 
By  the definition of $v$,
the mapping $$z\mapsto  v(z\otimes r)$$ is bounded   from $(L_p+L_2) (\mathcal{M}_1,\tau_1)$ to $L_p(\mathcal{N}_{\mathcal{M}\bar{\otimes} \mathbb{M}_2(\mathbb{C})},\sigma_{\mathcal{M}\bar{\otimes}  \mathbb{M}_2(\mathbb{C})})$.

Let $x,y\in (L_p+L_2) (\mathcal{M}_1,\tau_1  )$ be such that $|x|$ and $|y|$ are equimeasurable. 
Let $x=u_x|x|$ and $y=u_y|y|$ be the polar decompositions. 
Define 
$$x_n:= u_x \sum_{k=0}^{n \cdot 2^ n  } \frac{k}{2^ n} e^{|x|}\left(
\frac{k}{2^ n},\frac{k+1}{2^ n}\right], ~ y_n:= u_y \sum_{k=0}^{n \cdot 2^ n } \frac{k}{2^ n} e^{|y|}\left(
\frac{k}{2^ n},\frac{k+1}{2^ n}\right].$$
In particular, $x_n$ and $y_n$ have $\tau$-finite left and right support projections, and 
$$x_n\to x,~y_n\to y$$
  in $(L_p+L_2)(\mathcal{M}_1,\tau_1  )$ as $n\to\infty.$ 
  Hence, 
$$v(x_n\otimes r)\to v(x\otimes r)\mbox{ and } v(y_n\otimes r)\to v(y\otimes r)$$ 
as $n\to \infty $
in $L_p(\mathcal{N}_{\mathcal{M}\bar{\otimes} \mathbb{M}_2(\mathbb{C})},\sigma_{\mathcal{M}\bar{\otimes} \mathbb{M}_2(\mathbb{C})}).$ In particular,
$$\mu(v(x_n\otimes r))\to\mu(v(x\otimes r)),\quad \mu(v(y_n\otimes r))\to\mu(v(y\otimes r)),\quad n\to\infty,$$
in $L_p(0,1)$ (see e.g. \cite[Theorem 5.2]{CS94} or \cite[Proposition 52(iii)]{DP2}). 
Noting that $\mu (x_n\otimes r)$ and $\mu(y_n\otimes r)$ are equimeasurable for every $n\geq0,$ 
and appealing to   Lemma~\ref{lemmav1}, we obtain 
$$\mu(v(x_n\otimes r))=\mu(v(y_n\otimes r)) $$  for every $n\geq0,$
which implies that  
$$ \mu(v(x\otimes r))= \mu(v(y\otimes r)).$$
This proof is complete. 
\end{proof}

Now, we are ready to prove the main result of this paper. 
 
\begin{proof}[Proof of Theorem \ref{main}]
 Set $d_p(t)=t^{-\frac1p}$, $t\in (0,\infty )$.
Below, we present the proofs for  the case of semifinite von Neumann algebras and the case of general von Neumann algebras separately. 
The second case is clearly enough for our purpose to prove Theorem \ref{main}. However,
  the proof for the first case given below allows us to obtain a  smaller algebra $\cN$. 
 \begin{enumerate}
   \item   Assume that $\cM_0$ is semifinite.
Without loss of generality, we may assume that $\cM_0$ is atomless\cite[Lemma 2.3.18]{LSZ}. 

 Let $$\mathcal{M}_1=\mathcal{M}_0\bar{\otimes}L_{\infty}(0,\infty)\bar{\otimes}L_{\infty}(0,1) \mbox{ and } \tau_1=\tau_0\otimes\int\otimes\int.$$

   \item    Assume that $\cM_0$ is a general von Neumann algebra (in particular, type $III$).
Let $\mathfrak{M}_0 $ be the semifinite von Neumann algebra defined in Section \ref{Haagerup} such that 
$L_p(\cM_0)$ is isometric to a subspace of $L_{p,\infty}(\mathfrak{M}_0,\tau_0)$ with   
 \begin{align}\label{Ha0}
 \mu(t;x)\stackrel{\eqref{haLp}}{=}\norm{x}_{L_p(\cM_0)}\cdot t^{-\frac1p},~ \forall x\in L_p(\cM_0), ~\forall~  t>0. 
 \end{align} 
 Let $$\mathcal{M}_1=\mathfrak{M}_0\bar{\otimes}L_{\infty}(0,\infty)\bar{\otimes}L_{\infty}(0,1) \mbox{ and } \tau_1=\tau_0\otimes\int\otimes\int.$$ 
 \end{enumerate}

Let   $K_{\cM_1\bar{\otimes} \mathbb{M}_2}$ 
 be the Kruglov operator from 
  $L_1\left(\cM_1 \bar{\otimes} \mathbb{M}_2,\tau_1\otimes {\rm Tr}\right)$ into $L_1\left(\cN_{\cM_1 \bar{\otimes}   \mathbb{M}_2},\sigma_{\cM\bar{\otimes} \mathbb{M}_2}\right)$. By Proposition \ref{propbounded} (by taking $\cM= \cM_0\bar{\otimes}L_{\infty}(0,\infty)  \bar{\otimes} \mathbb{M}_2 $ or  $\cM= \mathfrak{M}_0\bar{\otimes}L_{\infty}(0,\infty)  \bar{\otimes} \mathbb{M}_2 $),  the operator $K_{\cM_1 \bar{\otimes} \mathbb{M}_2 }(\cdot \otimes r)  : L_1 \left(\mathcal{M},\tau_0\otimes \int \otimes {\rm Tr}  \right)\to  L_1\left(\cN_{\cM_1 \bar{\otimes}   \mathbb{M}_2},\sigma_{\cM_1\bar{\otimes} \mathbb{M}_2}\right) $\footnote{Note that $\cM_1 \bar{\otimes}\mathbb{M}_2  =\cM\bar{\otimes}L_\infty (0,1) $  and $K_{\cM_1 \bar{\otimes} \mathbb{M}_2 }=K_{\cM \bar{\otimes} L_\infty (0,1) }$.  }  extends to 
 \begin{itemize}
   \item  an  isomorphic embedding from   $(L_{p/2}+L_2) \left(\mathcal{M},\tau_0\otimes \int \otimes {\rm Tr}  \right)$ into $L_{p/2}\left(\cN_{\cM_1 \bar{\otimes}   \mathbb{M}_2},\sigma_{\cM_1\bar{\otimes} \mathbb{M}_2}\right)$, 
   \item an  isomorphic embedding from $ L_2  \left(\mathcal{M} ,\tau_0\otimes \int \otimes {\rm Tr}  \right)$ into $L_2\left(\cN_{\cM_1 \bar{\otimes}   \mathbb{M}_2},\sigma_{\cM_1 \bar{\otimes} \mathbb{M}_2} \right)$. 
 \end{itemize}

Let  $v(z)=K_{\mathcal{M}_1\bar{\otimes}  \mathbb{M}_2(\mathbb{C})}(z\otimes E_{01})$, $z\in (L_{p/2}+L_2) (\cM_1,\tau_1) $.
\begin{enumerate}
\item If $\cM_0$ is semifinite, then 
  for any $x\in L_p(\cM_0,\tau_0)$,
we have \begin{align*}
\mu\left(|x\otimes d_p|\right)=\mu\left(|x|\otimes d_p\right)&~\quad =\quad \mu\left(\mu(x)\otimes d_p\right)\\
&\stackrel{\tiny\mbox{\cite[p.211]{LT2}} }{=}\left\|\mu(x)\right\|_{L_p(0,\infty )}d_p=\left\|x\right\|_{L_p(\cM_0,\tau_0)}d_p.
\end{align*}
For an arbitrary projection $A\in \cM_0$ with trace $1$, we have$$\mu(x\otimes d_p) =\left\|x\right\|_{L_p(\cM_0,\tau_0)} \mu(d_p)=\left\|x\right\|_{L_p(\cM_0,\tau_0)} \mu(A\otimes d_p) . $$
Since $d_p\in  (L_{p/2}+L_2) (0,\infty )$ (see e.g. \cite[p.217]{Bennett_S}), it follows that  $\mu(x\otimes d_p ) \in  (L_{p/2}+L_2)   (0,\infty )$, and, by   Lemma~\ref{lemmav}, we have 
\begin{align}\label{iso1}
\mu\left(v(x\otimes d_p\otimes r)\right)= \left\|x\right\|_{L_p(\cM_0,\tau_0)} \mu\left(v(A\otimes d_p\otimes r)\right).
\end{align}  
\item If $\cM_0$ is a general von Neumann algebra (in particular, type $III$), then    for any $x\in L_p(\cM_0)$ and for
%we have \begin{align*}
%\mu\left(x\right) \stackrel{\eqref{Ha0}}{=} \left\|x\right\|_{L_p(\cM_0)}d_p.
%\end{align*}
  an  arbitrary projection $A\in \mathfrak{M} _0$ with trace $1$, 
  we have
  $$\mu\left(x \otimes \chi_{(0,1)}\right) =\mu\left(x  \right) \stackrel{\eqref{Ha0}}{=}  \left\|x\right\|_{L_p(\cM_0)} \cdot \mu(d_p)=\left\|x\right\|_{L_p(\cM_0 )} \cdot \mu(A\otimes d_p)
   . 
  $$ 
Since $d_p\in  (L_{p/2}+L_2) (0,\infty )$, it follows that
   $\mu\left(x \otimes \chi_{(0,1)}\right)  \in  (L_{p/2}+L_2) (0,\infty )$, and, by    Lemma~\ref{lemmav}, we have 
\begin{align}\label{iso2}
\mu\left(v(x\otimes\chi_{(0,1)} \otimes  r)\right)= \left\|x\right\|_{L_p(\cM_0)} \mu\left(v(A\otimes  d_p\otimes r)\right).
\end{align} 
\end{enumerate}

Since $L_{p,\infty}$ is an interpolation space between $L_{p/2}$ and $L_2$   
 whenever $0<  p<2$~\cite[Theorem 4.3]{Bennett_S}, 
it follows from \cite[Theorem 3.11]{Dirksen} that
 $K_{\cM
 _1\bar{\otimes} \mathbb{M}_2}$ is bounded from $L_{p,\infty}(\mathcal{M}_1 \bar{\otimes} \mathbb{M}_2,\tau_1\otimes {\rm Tr})$ to $L_{p,\infty}(\mathcal{N}_{\mathcal{M}_1\bar{\otimes} \mathbb{M}_2(\mathbb{C})},\sigma_{\mathcal{M}_1\bar{\otimes} \mathbb{M}_2(\mathbb{C})})$, 
 and therefore, 
the mapping $$z\mapsto v(z\otimes r)=K_{\mathcal{M}_1\bar{\otimes}  \mathbb{M}_2(\mathbb{C})}(z\otimes r\otimes E_{01}), ~z\in (L_{p/2}+L_2)\left(\mathfrak{M} _0\bar{\otimes}  L_\infty (0,\infty),\tau_0\otimes \int \right)$$ 
is bounded from $L_{p,\infty}\left(\mathfrak{M} _0\bar{\otimes}  L_\infty (0,\infty),\tau_0\otimes \int \right) $ to $L_{p,\infty}\left(\mathcal{N}_{\mathcal{M}_1 \bar{\otimes}  \mathbb{M}_2(\mathbb{C})},\sigma_{\mathcal{M}_1\bar{\otimes} \mathbb{M}_2(\mathbb{C})}\right).$
Since $E(0,1)$ contains $d_p$, it follows that $E(0,1)\supset L_{p,\infty}(0,1)$ and
$$\norm{v(A\otimes d_p\otimes r)}_{E}\stackrel{\tiny \mbox{\cite[Lemma I.3.3]{KPS}}}{\le} \norm{v(A\otimes d_p\otimes r)}_{L_{p,\infty}}<\infty.$$
 On the other hand, 
 since  $K_{\cM_1 \bar{\otimes} \mathbb{M}_2 }(\cdot \otimes r)$ is  an  isomorphic embedding from 
  $(L_{p/2}+L_2)\left (
  \mathcal{M} ,\tau_0\otimes \int \otimes {\rm Tr} \right )$ into $L_{p/2}\left(\cN_{\cM_1 \bar{\otimes}   \mathbb{M}_2},\sigma_{\cM_1\bar{\otimes} \mathbb{M}_2}\right) $, it follows that  
  $$v(A\otimes d_p\otimes r) =K_{\mathcal{M}_1\bar{\otimes}  \mathbb{M}_2(\mathbb{C})}(A\otimes d_p\otimes r  \otimes  E_{01})\ne 0.$$  By \eqref{iso1} (respectively, \eqref{iso2}), the mapping
$$x\mapsto \frac{1}{\norm{v(A\otimes d_p\otimes r)}_E}v(x\otimes d_p \otimes r) $$
(respectively, $$x\mapsto \frac{1}{\norm{v(A\otimes d_p\otimes r)}_E}v(x\otimes \chi_{(0,1)}\otimes r)\quad ) $$
is an isometry from $L_p(\cM_0)$ into $E\left(\mathcal{N}_{\mathcal{M}_1 \bar{\otimes}  \mathbb{M}_2(\mathbb{C})},\sigma_{\mathcal{M}_1\bar{\otimes} \mathbb{M}_2(\mathbb{C})}\right) $. 
\end{proof}
 
For a hyperfinite von Neumann algebra $\cM$, the crossed product constructed in Section \ref{Haagerup} is also hyperfinite \cite[Theorem XV.3.16]{Tak3} (see also \cite{AD}). 
Recall that for any hyperfinite von Neumann algebra, the algebra $\cN$ can be chosen to be hyperfinite, which is $*$-isomorphic to a subalgebra of $\cR$, the hyperfinite $II_1$-factor, see e.g. the proof of \cite[Theorem 5.2]{hrs}. 
We have the following consequence of Theorem \ref{main}, see also \cite[Corollary 34]{JSZ} for a similar argument. 
\begin{cor}
Let $0 <  p<2 $. 
 Let $\cM$ be a hyperfinite von Neumann algebra on a separable Hilbert space.  There exists an isometric embedding of $L_p(\mathcal{M},\tau)$ into $E(\mathcal{R},\tau_\cR )$ for the hyperfinite $II_1$-factor $\cR$ provided that 
$E(0,1)\supset L_{p,\infty}(0,1).$
\end{cor}

 \begin{rem}
   If one consider the self-adjoint part of $L_p(\cM,\tau)$, then there exists a noncommutative probability space $\cN$ such that $L_p(\cM,\tau)_h$ is isometric to a subspace of   $E(\cN,\sigma)_h$, the self-adjoint part of $E(\cN,\sigma)$. 
  Indeed, in Lemma \ref{lemmav1}, one may simply replace the operator $v$ (and Lemma \ref{lemmav}) with $K_\cM$.  
   In particular, if $\cM=L_\infty (0,\infty)$, then $\cN$ can be chosen to be $L_\infty (0,1)$ (by the construction of the first case in the proof of Theorem \ref{main} with $v$ replaced by $K_{\cM_1}$), which recovers \cite{BD,BDK} and Theorem \ref{JMST} above; if $\cM$ is hyperfinite, then $\cN$ can be chosen to be hyperfinite. 

This result should be compared with the main result in \cite{Kalton_R}, which shows that all surjective isometries on a rearrangement-invariant space $E(0,1)\ne L_2(0,1)$ are elementary.
 \end{rem}

\begin{rem}\label{Remark}
The same argument used in \cite[p.211]{LT2} yields that, 
for any $0 <  p <\infty $ and $E(0,\infty)\supset L_{p,\infty}(0,\infty )$, the space $L_{p}(\cM,\tau)$ is isometric to a subspace of $E \left(\cM\bar{\otimes}  L_\infty (0,\infty ), \tau\otimes \int \right)$, where 
 $\cM$ is  a semifinite von Neumann algebra  equipped with a semifinite faithful normal trace $\tau$.
Indeed, letting $d_p(t):=t^{-1/p}$, $0<t<\infty $, we have 
   $$\mu\left(|x\otimes d_p|\right)= \mu\left(\mu(x)\otimes d_p|\right) \stackrel{\tiny \mbox{\cite[p.211]{LT2}}}{=}\norm{\mu(x)}_{L_{p}(0,\infty )} d_p  =  \left\|x\right\|_{L_p(\cM,\tau)}d_p,$$
   which  implies that 
$$\norm{x\otimes d_p}_{E(\cM\bar{\otimes} L_\infty (0,\infty ))} = \norm{\mu(|x\otimes d_p|)}_{E(0,\infty )}   =\norm{\norm{x}_{L_{p}(\cM,\tau)}d_p }_{E(0,\infty )} .$$
Hence, $x\mapsto \frac{1}{\norm{d_p}_{E(0,\infty )}  } x\otimes d_p$ is an isometry from $L_p(\cM)$ into $E(\cM \bar{\otimes}  L_\infty )$. 
\end{rem}

\begin{rem}
  The condition that 
  $E(0,1)\supset L_{p,\infty }(0,1)$ in Theorem \ref{main}
   provides a very simple criterion for   the existence of a  noncommutative probability space  $(\cN,\sigma)$ such that $L_{p}(\cM,\tau)$ isometrically embedding into $E(\cN,\sigma)$.
   However, it is not a necessary condition even in the commutative setting.
  For example, since $L_p(0,\infty )$ (and any $L_p(\cM,\tau)$ affiliated with a semifinite von Neumann algebra $(\cM,\tau)$ on a separable Hilbert space, see e.g. \cite{Me}, \cite[Proposition 1.2]{S00} and \cite{DPS}), $1\le p<\infty$,  is separable, it follows from \cite[Theorem 2.5.7]{AK} that $L_p(0,\infty )$ (and $L_p(\cM,\tau)$) embeds isometrically into $\ell_\infty$ and therefore, into $L_\infty (0,1 )$. 
  However, $L_\infty (0,1)\not\supset L_{p,\infty }(0,1)$.
\end{rem}

\appendix
\section{The Kruglov operator on noncommutative $L_p$-space, $0<p<1$}
In \cite{JSZ}, the Kruglov operator was defined in the setting of noncommutative  $L_p$-space, $1\le p<\infty$. 
We extend Lemmas 35 to 38 from \cite{JSZ} to the setting of $0< p<1$. 
%The main tool is the so-called \emph{uniform submajorization} introduced in \cite{KS}:
%for $x,y\in S(\cM,\tau)$, $y$ is said to be uniformly majorized by $x$ (written $y \vartriangleleft x$) if there exists $\lambda \in \mathbb{N}$ such that
%$$\int_{\lambda a}^b \mu(s;y)ds \le \int_{a}^b \mu(s;x)ds ,~\forall ~ 0\le \lambda a \le b. $$
  If $\tau>0,$ the dilation operator $\sigma_{\tau}$ is defined by setting 
  $D_{\tau}x(s)=x(s/{\tau}),$ $s>0,$ in the case of the semi-axis. In the case of the interval $(0,1),$ the operator $D_{\tau}$ is defined by
 $$
 D_{\tau}x(s)=
 \begin{cases}
 x(s/\tau),& s\leq\min\{1,\tau\}\\
 0,& \tau<s\leq1.
 \end{cases}
 $$

Throughout this section, we always assume that $0<p<1$. Note that $\norm{\cdot}_{L_p}$ is not monotone with respect to the Hardy--Littlewood--Polya submajorization when $0<p<1$\cite[Lemma 25]{ASZ}. 

The following lemma is an analogue of \cite[Lemma 35]{JSZ}.
\begin{lem}\label{A35}
  For every $x\in   L_\infty (0,1)$, we have 
  $$\mu(Kx)\le   \mu\left( \bigoplus_{n=1}^\infty n D_{\frac{1}{e\cdot (n-1)!}} \mu(x) \right) . $$
\end{lem}
\begin{proof} 
Recall that 
$$Kx=\sum_{n=1}^\infty  \chi_{_{A_n}} \otimes  \sum_{m=1}^n \chi_{(0,1)}^{\otimes (m-1)} \otimes x \otimes \chi_{(0,1)}^{\otimes \infty 
}, $$
where $A_n$, $n\ge 0$, are pairwise disjoint sets with $m(A_n) = \frac{1}{e\cdot n!} $ for all $n\ge 0$. 

It follows from \cite[Lemma 2.3.16(a)]{LSZ} that 
$$\mu\left(t;\sum_{m=1}^n 1^{\otimes (m-1)} \otimes x\otimes {\bf 1}^{\otimes \infty } \right)
\le n\mu\left(
\frac{t}{n}; x\right), ~ \forall n\ge 1, ~t>0,$$
and, 
therefore, 
$$\mu\left(t; 
\chi_{A_n}\otimes \left(\sum_{m=1}^n 1^{\otimes (m-1)} \otimes x\otimes {\bf 1}^{\otimes \infty }    \right)  
\right)
\le  
  nD_{m(A_n)}  \mu\left(\frac{t}{n};x\right)  =n D_{\frac{1}{e\cdot (n-1)!}}\mu(x)  .  $$
Hence, we have 
$$\mu(Kx)   \le  \mu\left( \bigoplus_{n=1}^\infty n D_{\frac{1}{e\cdot (n-1)!}}\mu(x)  \right) .$$ 
This completes the proof. 
\end{proof}

\begin{lem}\label{A36}
  If $0<p<1$ and $x_k \in  L_\infty (0,1)$, $1\le k \le n$, then 
  $$\norm{\oplus _{k=0}^\infty Kx_k }_{(L_p+L_2)(0,\infty)} \sim \norm{\oplus _{k=0}^\infty  x_k }_{(L_p+L_2)(0,\infty)} .$$
\end{lem}
\begin{proof}

  By Lemma \ref{A35}, we have 
  $$\mu\left( \bigoplus _{k=0}^\infty Kx_k \right)   \le  \mu\left(  \bigoplus _{k=0}^\infty \left(\bigoplus _{n=1}^\infty n D_{\frac{1}{e\cdot (n-1)!}} \mu(x_k)  \right)  \right)= \mu\left( \bigoplus_{n=1}^\infty n\left( \bigoplus _{k=0}^\infty   D_{\frac{1}{e\cdot (n-1)!}} \mu(x_k )  \right) \right) . $$
  %Since the quasi-norm $\norm{\cdot}_{(L_p+L_2)(0,\infty)}$ is monotone with respect to uniform submajorization (up to a constant\cite[Corollary 14]{S14}), i
  It follows from the quasi-triangular inequality of $\norm{\cdot}_{(L_p+L_2)(0,\infty)}$ that 
  \begin{align*}
&~\quad   \norm{\bigoplus _{k=0}^\infty Kx_k}_{(L_p+L_2)(0,\infty)} \\&\le  \sum_{n=1}^\infty n C_p^n    \norm{ \bigoplus _{k=0}^\infty  D_{\frac{1}{e\cdot (n-1)!}} \mu(x_k) }_{(L_p+L_2)(0,\infty)} \\ 
  &  = \sum_{n=1}^\infty n C_p^n  \norm{  D_{\frac{1}{e\cdot (n-1) !}} \left( \bigoplus _{k=0}^\infty \mu(x_k)\right) }_{(L_p+L_2)(0,\infty)} \\
  &\le  \left(     \sum_{n=1}^\infty n C_p^n \norm{D_{\frac{1}{e\cdot (n-1) !}}}_{L_p+L_2 \to L_p+L_2} \right)  \norm{  \bigoplus _{k=0}^\infty x_k   }_{(L_p+L_2)(0,\infty)}  ,
  \end{align*}
  where $C_p$ stands for the modulus of concavity of the quasi-norm $\norm{\cdot}_{L_p+L_2}$. 
Taking into account that 
$$\norm{D_u}_{L_p+L_2 \to L_p+L_2} \le u^{1/p}, ~0<u\le 1,$$
we infer that 
$$\left(     \sum_{n=1}^\infty n C_p^n \norm{D_{\frac{1}{e\cdot n!}}}_{L_p+L_2 \to L_p+L_2} \right)
\le \sum_{n=1}^\infty \frac{n C_p^n }{(e\cdot (n-1) !)^{1/p}}   <\infty. 
$$
Therefore, there exists a constant $  {\rm Const}<\infty $
such that 
$$ \norm{\bigoplus _{k=0}^\infty Kx_k}_{(L_p+L_2)(0,\infty)}  \le  {\rm Const} \cdot \norm{  \bigoplus _{k=0}^\infty x_k   }_{(L_p+L_2)(0,\infty)}   . $$

On the other hand, by the definition of $K$, we have 
$$\mu(Kx_k )\ge \mu(\chi_{_{A_1}} \otimes x_k \otimes {\bf 1}^\infty ) =D_{m(A_1)} \mu(x_k)=D_{1/e}\mu(x_k), ~1\le k\le n.$$
Observing that 
$$\norm{D_u x}_{(L_p+L_2)(0,\infty)} \ge u^{1/p} \norm{x}_{(L_p+L_2)(0,\infty)}, ~\forall x\in (L_p+L_2)(0,\infty ), ~0<u\le 1, $$
we conclude that 
\begin{align*}
\norm{\oplus _{k=0}^\infty K  x_k }_{(L_p+L_2)(0,\infty)} &\ge \norm{\oplus _{k=0}^\infty  D_{1/e}\mu(x_k)  }_{(L_p+L_2)(0,\infty)}  \\
& =\norm{D_{1/e}\left( \oplus _{k=0}^\infty  \mu(x_k)\right)   }_{(L_p+L_2)(0,\infty)}\\
&\ge  \left(\frac1e\right)^{\frac1p}\norm{  \oplus _{k=0}^\infty  \mu(x_k)   }_{(L_p+L_2)(0,\infty)} . 
\end{align*}
This completes the proof. 
\end{proof}
%We denote by $\norm{\cdot}_{\mathcal{L}_p+\mathcal{L}_2}$ the standard quasi-norm on $(L_p+L_2)(\cR\bar{\otimes}B(\cH) )$. 
Recall that the Johnson--Schechtman inequality holds in the setting of $L_p$-space, $0<p<\infty$\cite{JS89}. The  argument in \cite[Lemma 37]{JSZ} yields the following.
\begin{lem}\label{A37}
Let $(\cM,\tau)$ be an atomless noncommutative probability space.
  Let $0<p\le 2$ and let $x_k\in \cM$, $1\le k\le n$, be independent symmetrically distributed random variables, in the sense of the definition given in  \cite[Section 2.8]{JSZ}. 
  Then, 
  $$\norm{\sum_{k=1}^n x_k}_{L_p(\cM,\tau )}\sim \norm{\bigoplus_{k=1}^n  \mu(x_k)}_{(L_p+L_2)(0,\infty )}.$$
\end{lem}
Recall that $K_\cM$ is bounded from $(L_1\cap L_2 )(\cM,\tau) $ into $L_2(\cN,\sigma)$. 
The following proposition shows that $K_{\cM\bar{\otimes} L_\infty (0,1)}(\cdot \otimes r )$ is bounded from $L_2(\cM,\tau)$($=(L_2+L_2)(\cM,\tau)$) into $L_2(\cN_{\cM\bar{\otimes}L_\infty(0,1)},\sigma_{\cM\bar{\otimes}L_\infty(0,1)})$.
The argument used in the proof of \cite[Lemma 38]{JSZ},  with \cite[Lemma 36]{JSZ}   (respectively, \cite[Lemma 37]{JSZ}) replaced by Lemma \ref{A36} (respectively, Lemma \ref{A37}),  yields the following proposition. For the sake of completeness, we include a full proof below.

\begin{prop}\label{propbounded} Let $\cM$ be an atomless   von Neumann algebra on a separable Hilbert space, equipped with a semifinite faithful normal trace $\tau$.  
For $0<  p\le 2$ and $x\in L_p(\cM,\tau)$ whose   right support  is  $\tau$-finite, we have 
$$\norm{K_{\cM\bar{\otimes} L_\infty (0,1)}(x\otimes r ) }_{L_p(\cN_{\cM\bar{\otimes} L_\infty (0,1)},\sigma_{\cM\bar{\otimes} L_\infty (0,1)})}  \sim \norm{x}_{(L_p +L_2) (\cM,\tau)}. $$
Consequently, the mapping $x\mapsto K_{\cM\bar{\otimes} L_\infty (0,1)}   (x\otimes r) $ extends to an isomorphic embedding of $(L_p +L_2) (\cM,\tau)$ into a noncommutative $L_p$-space over a noncommutative probablity space $\left(
\cN_{\cM\bar{\otimes} L_\infty (0,1)},\sigma_{ \cM\bar{\otimes} L_\infty (0,1) }\right)$.
\end{prop}
\begin{proof}
Let  $K_{\cM\bar{\otimes} L_\infty (0,1)}(x\otimes r )$  be the Kruglov operator  from  $L_1(\cM\bar{\otimes} L_\infty (0,1) ,\tau\otimes \int)$ into $L_1(\cN_{\cM\bar{\otimes} L_\infty (0,1)},\sigma_{\cM\bar{\otimes} L_\infty (0,1)})$ for some finite von Neumann algebra $\cN_{\cM\bar{\otimes} L_\infty (0,1)}$ equipped with a faithful normal tracial state $\sigma_{ \cM\bar{\otimes} L_\infty (0,1) }$.

(1) Assume that $x$ is a self-adjoint operator whose  right supports are $\tau$-finite. 
One may select $x_k=x^*_k \in \cM $, $1\le k \le n $, such that $x_k x_{j}=0$ for $k\ne j $, $x=\sum_{k=1}^n x_k$ and $(\tau\otimes \int)  ({\rm supp} x_k)\le 1$ for every $k$. 
Since the random variables $x_k \otimes r$, $1\le k \le n$, are symmetrically distributed, it follows from \cite[Corollary 33 (v) and (iv)]{JSZ} that the random variables $K_{\cM\bar{\otimes} L_\infty (0,1)} (x_k\otimes r )$, $1\le k \le n $, are independent and symmetrically distributed. 
By Lemma \ref{A37} above, we have 
\begin{align*}
&~\quad \norm{K_{\cM\bar{\otimes} L_\infty (0,1)} (x \otimes r )}_{L_p(\cN_{\cM\bar{\otimes} L_\infty (0,1)},\sigma_{\cM\bar{\otimes} L_\infty (0,1)})} \\
&=\norm{\sum_{k=1}^n K_{\cM\bar{\otimes} L_\infty (0,1)} (x_k\otimes r )}_{L_p(\cN_{\cM\bar{\otimes} L_\infty (0,1)},\sigma_{\cM\bar{\otimes} L_\infty (0,1)})}\\
&\sim \norm{\bigoplus_{k=1}^n \mu\left(K_{\cM\bar{\otimes} L_\infty (0,1)} (x_k\otimes r )\right)}_{(L_p+L_2)(0,\infty )}. 
\end{align*}
For every $1\le k \le n$, select a function $y_k \in L_\infty (0,1)$ which is equi-measurable with $x_k \otimes r $. 
By \cite[Corollary 33(i)]{JSZ}, we have that $K_{\cM\bar{\otimes} L_\infty (0,1)} (x_k\otimes r )$ is equimeasurable with $Ky_k$. 
By Lemma \ref{A36}, we have 
\begin{align*}
\norm{\bigoplus_{k=1}^n \mu\left(K_{\cM\bar{\otimes} L_\infty (0,1)} (x_k\otimes r )\right)}_{(L_p+L_2)(0,\infty )}
&=\norm{\oplus_{k=1}^n Ky_k  }_{(L_p+L_2)(0,\infty )}\\
 &\stackrel{L.\ref{A36}}{\sim} \norm{\bigoplus _{k=1}^n y_k}_{(L_p+L_2)(0,\infty )}  \\
&~=  \norm{\bigoplus _{k=1}^n \mu( x_k)}_{(L_p+L_2)(0,\infty )}  . 
\end{align*}
Hence, we have 
\begin{align*}
\norm{  K_{\cM\bar{\otimes} L_\infty (0,1)}  (x_k\otimes r) }_{L_p(\cN_{\cM\bar{\otimes} L_\infty (0,1)},\sigma_{\cM\bar{\otimes} L_\infty (0,1)})} 
&  \sim  \norm{\bigoplus _{k=1}^n \mu( x_k)}_{(L_p+L_2)(0,\infty )} \\
& =\norm{x}_{(L_p+L_2)(\cM,\tau)}. 
\end{align*}

(2)
For a not necessarily self-adjoint finitely supported operator $x\in  L_p(\cM,\tau)$, we have 
\begin{align*}
 \norm{x}_{ (L_p+L_2)(\cM,\tau)}&\sim \norm{\Re x}_{ (L_p+L_2)(\cM,\tau)}+\norm{\Im x}_{ (L_p+L_2)(\cM,\tau)}\\&\sim \norm{K_{\cM\otimes L_\infty (0,1)}(\Re x \otimes r )}_{_{L_p\left(\cN_{\cM\bar{\otimes} L_\infty (0,1)},\sigma_{\cM\bar{\otimes} L_\infty (0,1)}\right)}}  \\
 &\qquad +\norm{K_{\cM\otimes L_\infty (0,1)} (\Im  x \otimes r )}_{_{L_p\left(\cN_{\cM\bar{\otimes} L_\infty (0,1)},\sigma_{\cM\bar{\otimes} L_\infty (0,1)}\right)}} \\
& \sim 
\norm{K_{\cM\otimes L_\infty (0,1)}(  x \otimes r )}_{_{L_p\left(\cN_{\cM\bar{\otimes} L_\infty (0,1)},\sigma_{\cM\bar{\otimes} L_\infty (0,1)}\right)}} . 
 \end{align*}
 This completes the proof. 
  \end{proof}

\end{document}